\documentclass[11pt,thmsa,arabtex]{article}
\usepackage{amssymb}
\usepackage{amsfonts}
\usepackage{amsmath}
\usepackage{graphicx}
\usepackage{hyperref}
\usepackage{cite}
\usepackage{ntheorem}

\newtheorem*{theorem-non}{Theorem}
\newtheorem{theorem}{Theorem}

\newtheorem{case}{Case}[section]

\newtheorem{definition}{Definition}[section]
\newtheorem{notation}{Notation}[section]

\newenvironment{proof}[1][Proof]{\noindent\textbf{#1.} }{\ \rule{0.5em}{0.5em}}

\numberwithin{theorem}{section}
\numberwithin{equation}{section}

\begin{document}

\title{{\Large \textbf{Equiform Differential Geometry of Curves\\
in Minkowski Space-Time}}}
\author{\emph{\textbf{H. S. Abdel-Aziz}}, \emph{\textbf{M. Khalifa Saad}}%
\thanks{
~E-mail address:~mohamed\_khalifa77@science.sohag.edu.eg} \ and \emph{%
\textbf{A. A. Abdel-Salam}} \\
{\small \emph{Dept. of Math., Faculty of Science, Sohag Univ., 82524 Sohag,
Egypt}}}
\date{}
\maketitle

\textbf{Abstract.} In this paper, we establish equiform differential
geometry of space and timelike curves in $4$-dimensional Minkowski space $%
E_{1}^{4}$. We obtain some conditions for these curves. Also, general helices with respect to their equiform curvatures are characterized.%

\section{Introduction}

In the case of a differentiable curve, at each point a tetrad of mutually
orthogonal unit vectors (called tangent, normal, first binormal and second
binormal) was defined and constructed, and the rates of change of these
vectors along the curve define the curvatures of the curve in Minkowski
space-time \cite{3}. The corresponding Frenet's equations for an arbitrary
curve in the Minkowski space $E_{1}^{4}$ are given in \cite{5,6}. Although
the authors wrote these important equations for the spacelike curves in
Minkowski space-time, which is simultaneously the geometry of special
relativity and the geometry induced on each fixed tangent space of an
arbitrary Lorentzian manifold, there was not a method to calculate all
Frenet apparatus of an arbitrary spacelike curve. In the light of this,
analogy with the method in \cite{9,4}, they have a method to calculate
Frenet apparatus of the spacelike curves with non-null frame vectors
according to signature $(-,+,+,+)$. Among all curves, space and timelike
curves have special emplacement regarding their properties and applicability
in other sciences. Because of this they deserve especial attention in
Euclidean as well as in other geometries. Besides the Euclidean geometry, a
palette of new geometries has been developed over the last two centuries and
some properties of curves and surfaces are more emphasized in newly
developed non-Euclidean geometries than in the Euclidean. Among these
non-Euclidean geometries there is also the Minkowski geometry which
represent ambient geometries in which we investigate the geometry of space
and timelike curves. We define equiform spacelike and timelike curves and
give a characterization of these curves in Minkowski space-time. We found
motivation for this work in \cite{1,2,8}, where the authors considered
characterizations of general helices in the Minkowski space-time,
pseudo-Galilean, double isotropic and Galilean space. We examin a similar
problem in the equiform differential geometry of the Minkowski space-time
for space and timelike curves and these curves will be called equiform
spacelike and timelike curves. The main goal of this article is to define,
describe and characterize equiform spacelike and timelike curves in
Minkowski $4$-space $E_{1}^{4}$.

\section{Preliminaries}

Let $E^{4}=\{\left( x_{1},x_{2},x_{3},x_{4}\right) \left\vert
x_{1},x_{2},x_{3},x_{4}\in \mathcal{R}\right. \}$ be a 4-dimensional vector
space. For any two vectors $x=\left( x_{1},x_{2},x_{3},x_{4}\right) ,\
y=\left( y_{1},y_{2},y_{3},y_{4}\right) $ in $E^{4}$, the pseudo scalar
product of $x$ and $y$ is defined by \ $\langle x,y\rangle
=-x_{1}y_{1}+x_{2}y_{2}+x_{2}y_{2}+x_{2}y_{2}.$ We call $\left(
E^{4},~\langle ~,~\rangle \right) $ a Minkowski $4$-space and denote it by~$%
E_{1}^{4}.$ We say that a vector $x$ in $E_{1}^{4}\backslash \left\{
0\right\} $ is a spacelike vector, a lightlike vector or a timelike vector
if $\langle x,y\rangle $ is positive, zero, negative respectively. The norm
of a vector $x$ $\in E_{1}^{4}$ is defined by $\left\Vert x\right\Vert =%
\sqrt{\left\vert \langle x,x\rangle \right\vert }$. For any two vectors $a;b$
in $E_{1}^{4}$, we say that $a$ is pseudo-perpendicular to $b$ if $\langle
a,b\rangle =0.$ Let $\alpha :I\subset \mathcal{R}\rightarrow E_{1}^{4}$ be
an arbitrary curve in $E_{1}^{4}$, we say that a curve $\alpha $ is a
spacelike curve if $\langle \dot{\alpha}(t),\dot{\alpha}(t)\rangle >0$ for
any $t\in I$. The arclength of a spacelike curve $\gamma $ measured from $%
\alpha (t_{0})\left( t_{0}\in I\right) $ is
\begin{equation}
s(t)=\int\limits_{t_{0}}^{t}\left\Vert \dot{\alpha}(t)\right\Vert dt.
\end{equation}%
Then a parameter $s$ is determined such that $\left\Vert \alpha ^{^{\prime
}}(s)\right\Vert =1$, where $\alpha ^{^{\prime }}(s)=d\alpha /ds.$
Consequently, we say that a spacelike curve $\alpha $ is parameterized by
arclength if $\left\Vert \alpha ^{^{\prime }}(s)\right\Vert =1.$ Throughout
the rest of this paper $s$ is assumed arclength parameter.\newline
For any $x_{1},x_{2},x_{3}\in E_{1}^{4},$ we define a vector $x_{1}\times
x_{2}\times x_{3}$ by
\begin{equation}
x_{1}\times x_{2}\times x_{3}=\left\vert
\begin{array}{cccc}
-e_{1} & e_{2} & e_{3} & e_{4} \\
x_{1}^{1} & x_{1}^{2} & x_{1}^{3} & x_{1}^{4} \\
x_{2}^{1} & x_{2}^{2} & x_{2}^{3} & x_{2}^{4} \\
x_{3}^{1} & x_{3}^{2} & x_{3}^{3} & x_{3}^{4}%
\end{array}%
\right\vert ,
\end{equation}%
where $x_{i}=\left( x_{i}^{1},x_{i}^{2},x_{i}^{3},x_{i}^{4}\right) .$ Let $%
\alpha :I\rightarrow E_{1}^{4}$ be a spacelike curve in $E_{1}^{4}$. Then we
can construct a pseudo-orthogonal frame $\left\{ \mathbf{t}(s),\mathbf{n}(s),%
\mathbf{b}_{1}(s),\mathbf{b}_{2}(s)\right\} ,$ which satisfies the following
Frenet-Serret type formulae of $E_{1}^{4}$ along $\alpha .$
\begin{equation}
\left[
\begin{array}{c}
\mathbf{t} \\
\mathbf{n} \\
\mathbf{b}_{1} \\
\mathbf{b}_{2}%
\end{array}%
\right] ^{\prime }=\left[
\begin{array}{cccc}
0 & \kappa _{1} & 0 & 0 \\
\mu _{1}\kappa _{1} & 0 & \mu _{2}\kappa _{2} & 0 \\
0 & \mu _{3}\kappa _{2} & 0 & \mu _{4}\kappa _{3} \\
0 & 0 & \mu _{5}\kappa _{3} & 0%
\end{array}%
\right] \left[
\begin{array}{c}
\mathbf{t} \\
\mathbf{n} \\
\mathbf{b}_{1} \\
\mathbf{b}_{2}%
\end{array}%
\right] ,
\end{equation}%
where $\kappa _{1},\kappa _{2}$ and $\kappa _{3}$ are respectively, first,
second and third curvature of the spacelike curve $\alpha $, and we have
\cite{8}
\begin{eqnarray}
\kappa _{1}(s) &=&\left\Vert \alpha ^{^{\prime \prime }}(s)\right\Vert ,
\notag \\
\mathbf{n}(s) &=&\frac{\alpha ^{^{\prime \prime }}(s)}{\kappa _{1}(s)},
\notag \\
\mathbf{b}_{1}(s) &=&\frac{\mathbf{n}^{\prime }(s)+\mu _{1}\kappa _{1}(s)%
\mathbf{t}(s)}{\left\Vert \mathbf{n}^{\prime }(s)+\mu _{1}\kappa _{1}(s)%
\mathbf{t}(s)\right\Vert },  \notag \\
\mathbf{b}_{2}(s) &=&\mathbf{t}(s)\times \mathbf{n}(s)\times \mathbf{b}%
_{1}(s).
\end{eqnarray}%
Denote by $\left\{ \mathbf{t}(s),\mathbf{n}(s),\mathbf{b}_{1}(s),\mathbf{b}%
_{2}(s)\right\} ~$the moving Frenet frame along the spacelike curve $\alpha $%
, where $s$ is a pseudo arclength parameter. Then $\mathbf{t}(s)$ is a
spacelike tangent vector and due to the causal character of the principal
normal vector $\mathbf{n}(s)$ and the binormal vector $\mathbf{b}_{1}(s)$,
we have the following Frenet formulas \cite{5}:

\begin{case}
If $\mathbf{n}$ is spacelike vector, then $\mathbf{b}_{1}$can have two
causal characters.\newline
\textit{Case 2.1.1:} if $\mathbf{b}_{1}$ is spacelike vector, then $\mu
_{i}\left( 1\leq i\leq 5\right) $ read%
\begin{equation*}
\mu _{1}=\mu _{3}=-1,\mu _{2}=\mu _{4}=\mu _{5}=1.
\end{equation*}%
And $\mathbf{t},\mathbf{n},\mathbf{b}_{1}$ and $\mathbf{b}_{2}$ are mutually
orthogonal vectors satisfying equations:%
\begin{equation*}
\langle \mathbf{t},\mathbf{t}\rangle =\langle \mathbf{n},\mathbf{n}\rangle
=\langle \mathbf{b}_{1},\mathbf{b}_{1}\rangle =1,\ \langle \mathbf{b}_{2},%
\mathbf{b}_{2}\rangle =-1.
\end{equation*}
\textit{Case 2.1.2}: if $\mathbf{b}_{1}$ is timelike vector, then $\mu
_{i}\left( 1\leq i\leq 5\right) $ read%
\begin{equation*}
\mu _{1}=-1,\mu _{2}=\mu _{3}=\mu _{4}=\mu _{5}=1.
\end{equation*}%
The vectors $\mathbf{t},\mathbf{n},\mathbf{b}_{1}$ and $\mathbf{b}_{2}$
satisfy the conditions:%
\begin{equation*}
\langle \mathbf{t},\mathbf{t}\rangle =\langle \mathbf{n},\mathbf{n}\rangle
=\langle \mathbf{b}_{2},\mathbf{b}_{2}\rangle =1,\langle \mathbf{b}_{1},%
\mathbf{b}_{1}\rangle =-1.
\end{equation*}
\end{case}

\begin{case}
If $\mathbf{n}$ is timelike vector. Then $\mu _{i}\left( 1\leq i\leq
5\right) $ read%
\begin{equation*}
\mu _{1}=\mu _{2}=\mu _{3}=\mu _{4}=1,\mu _{5}=-1.
\end{equation*}%
For the vectors $\mathbf{t},\mathbf{n},\mathbf{b}_{1}$ and $\mathbf{b}_{2}$
, we have%
\begin{equation*}
\langle \mathbf{t},\mathbf{t}\rangle =\langle \mathbf{b}_{1},\mathbf{b}%
_{1}\rangle =\langle \mathbf{b}_{2},\mathbf{b}_{2}\rangle =1,\langle \mathbf{%
n},\mathbf{n}\rangle =-1.
\end{equation*}
\end{case}

Now, let $\gamma $ be a timelike curve. Then $\mathbf{T}$ is timelike vector
and the Frenet equations have the form
\begin{equation}
\left[
\begin{array}{c}
\mathbf{T} \\
\mathbf{N} \\
\mathbf{B}_{1} \\
\mathbf{B}_{2}%
\end{array}%
\right] ^{\prime }=\left[
\begin{array}{cccc}
0 & \bar{\kappa}_{1} & 0 & 0 \\
\bar{\kappa}_{1} & 0 & \bar{\kappa}_{2} & 0 \\
0 & -\bar{\kappa}_{2} & 0 & \bar{\kappa}_{3} \\
0 & 0 & -\bar{\kappa}_{3} & 0%
\end{array}%
\right] \left[
\begin{array}{c}
\mathbf{T} \\
\mathbf{N} \\
\mathbf{B}_{1} \\
\mathbf{B}_{2}%
\end{array}%
\right] ,
\end{equation}%
where $\mathbf{T},\mathbf{N},\mathbf{B}_{1}$ and $\mathbf{B}_{2}$ satisfy
the following equations%
\begin{equation*}
\langle \mathbf{N},\mathbf{N}\rangle =\langle \mathbf{B}_{1},\mathbf{B}%
_{1}\rangle =\langle \mathbf{B}_{2},\mathbf{B}_{2}\rangle =1,\langle \mathbf{%
T},\mathbf{T}\rangle =-1.
\end{equation*}
The functions $\bar{\kappa}_{1},\bar{\kappa}_{2}$ and $\bar{\kappa}_{3}$ are
the curvatures of $\gamma $\cite{10}.

\section{Equiform differential geometry of space and timelike curves}

\textbf{Spacelike curves:}\newline
Let $\alpha :I\rightarrow E_{1}^{4}$ be \textbf{a spacelike curve}. We
define the equiform parameter of $\alpha \left( s\right) $ by%
\begin{equation}
\sigma =\int \frac{ds}{\rho }=\int \kappa _{1}ds
\end{equation}%
where $\rho =\frac{1}{\kappa _{1}},$ is the radius of curvature of the curve
$\alpha $. It follows%
\begin{equation}
\frac{ds}{d\sigma }=\rho .
\end{equation}%
Let $h$ is a homothety with the center in the origin and the coefficient $%
\lambda .$ If we put $\alpha ^{\ast }=h\left( \alpha \right) ,$ then it
follows
\begin{equation}
s^{\ast }=\lambda s,\ \text{and }\rho ^{\ast }=\lambda \rho ,
\end{equation}%
where $s^{\ast }$ is the arclength parameter of $\alpha ^{\ast }$ and $\rho
^{\ast }$ the radius of curvature of $\alpha ^{\ast }$. Hence $\sigma $ is
an equiform invariant parameter of $\alpha .$

\begin{notation}
Let us note that $\kappa _{1},\kappa _{2}$ and $\kappa _{3}$ are not
invariants of the homothety group, it follows $\kappa _{1}^{\ast }=\frac{1}{%
\lambda }\kappa _{1}$, \ $\kappa _{2}^{\ast }=\frac{1}{\lambda }\kappa _{2}$
and $\kappa _{3}^{\ast }=\frac{1}{\lambda }\kappa _{3}$.
\end{notation}

The vector%
\begin{equation}
\mathbf{V}_{1}=\frac{d\alpha \left( s\right) }{d\sigma },
\end{equation}%
is called a tangent vector of the curve $\alpha $ in the equiform geometry.
From $(3.2)$ and $(3.4)$, we get%
\begin{equation}
\mathbf{V}_{1}=\frac{d\alpha \left( s\right) }{d\sigma }=\rho \frac{d\alpha
\left( s\right) }{ds}=\rho \mathbf{t}.
\end{equation}%
Furthermore, we define the tri-normals by%
\begin{equation}
\mathbf{V}_{2}=\rho \mathbf{n},\ \ \ \mathbf{V}_{3}=\rho \mathbf{b}_{1},\ \
\mathbf{V}_{4}=\rho \mathbf{b}_{2}.
\end{equation}%
It is easy to check that the tetrahedron $\left\{ \mathbf{V}_{1},\mathbf{V}%
_{2},\ \mathbf{V}_{3},\ \mathbf{V}_{4}\right\} $ is an equiform invariant
tetrahedron of the curve $\alpha $ \cite{4}$.$ Now, we will find the
derivatives of these vectors with respect to $\sigma $ using by $\left(
3.2\right) ,\ \left( 3.4\right) $ and $\left( 3.6\right) \ $as follows:
\begin{equation*}
\mathbf{V}_{1}^{^{\prime }}=\frac{d}{d\sigma }\left( \mathbf{V}_{1}\right)
=\rho \frac{d}{ds}\left( \rho \mathbf{t}\right) =\dot{\rho}\mathbf{V}_{1}+%
\mathbf{V}_{2},
\end{equation*}%
where the derivative with respect to the arclength $s$ is denoted by a dot
and respect to $\sigma $ by a dash. Similarly, we obtain%
\begin{eqnarray}
\mathbf{V}_{2}^{^{\prime }} &=&\frac{d}{d\sigma }\left( \mathbf{V}%
_{2}\right) =\rho \frac{d}{ds}\left( \rho \mathbf{n}\right) =\mu _{1}\mathbf{%
V}_{1}+\dot{\rho}\mathbf{V}_{2}+\mu _{2}\left( \frac{\kappa _{2}}{\kappa _{1}%
}\right) \mathbf{V}_{3},  \notag \\
\mathbf{V}_{3}^{^{\prime }} &=&\frac{d}{d\sigma }\left( \mathbf{V}%
_{3}\right) =\rho \frac{d}{ds}\left( \rho \mathbf{b}_{1}\right) =\mu
_{3}\left( \frac{\kappa _{2}}{\kappa _{1}}\right) \mathbf{V}_{2}+\dot{\rho}%
\mathbf{V}_{3}+\mu _{4}\left( \frac{\kappa _{3}}{\kappa _{1}}\right) \mathbf{%
V}_{4},  \notag \\
\mathbf{V}_{4}^{^{\prime }} &=&\frac{d}{d\sigma }\left( \mathbf{V}%
_{4}\right) =\rho \frac{d}{ds}\left( \rho \mathbf{b}_{2}\right) =\mu
_{5}\left( \frac{\kappa _{3}}{\kappa _{1}}\right) \mathbf{V}_{3}+\dot{\rho}%
\mathbf{V}_{4},
\end{eqnarray}

\begin{definition}
The functions$\mathcal{\ K}_{i}:I\rightarrow \mathcal{R}\ \left(
i=1,2,3\right) $ defined by
\begin{equation}
\mathcal{K}_{1}=\dot{\rho},\ \mathcal{K}_{2}=\frac{\kappa _{2}}{\kappa _{1}}%
,\ \mathcal{K}_{3}=\frac{\kappa _{3}}{\kappa _{1}}
\end{equation}%
are called $i^{th}$equiform curvatures of the curve $\alpha .\ $These
functions $\mathcal{K}_{i}$ are differential invariant of the group of
equiform transformations, too.
\end{definition}

Thus, the formulas analogous to famous the Frenet formulas in the equiform
geometry of the Minkowski space $E_{1}^{4}$ have the following form:%
\begin{eqnarray}
\mathbf{V}_{1}^{^{\prime }} &=&\mathcal{K}_{1}\mathbf{V}_{1}+\mathbf{V}_{2},
\notag \\
\mathbf{V}_{2}^{^{\prime }} &=&\mu _{1}\mathbf{V}_{1}+\mathcal{K}_{1}\mathbf{%
V}_{2}+\mu _{2}\mathcal{K}_{2}\mathbf{V}_{3},  \notag \\
\mathbf{V}_{3}^{^{\prime }} &=&\mu _{3}\mathcal{K}_{2}\mathbf{V}_{2}+%
\mathcal{K}_{1}\mathbf{V}_{3}+\mu _{4}\mathcal{K}_{3}\mathbf{V}_{4},  \notag
\\
\mathbf{V}_{4}^{^{\prime }} &=&\mu _{5}\mathcal{K}_{3}\mathbf{V}_{3}+%
\mathcal{K}_{1}\mathbf{V}_{4}.
\end{eqnarray}

\begin{notation}
The equiform parameter $\sigma =\int \kappa _{1}(s)ds$ for closed curves is
called the total curvature, and it plays an important role in global
differential geometry of the Euclidean space. Also, the functions $\frac{%
\kappa _{2}}{\kappa _{1}}$ and $\frac{\kappa _{3}}{\kappa _{1}}$ have been
already known as conical curvatures and they also have interesting geometric
interpretation.
\end{notation}

Because of the equiform Frenet formulas $(3.9)$, the following equalities
regarding equiform curvatures can be given
\begin{eqnarray}
\mathcal{K}_{1} &=&\frac{1}{\rho ^{2}}\langle \mathbf{V}_{j}^{^{\prime }},%
\mathbf{V}_{j}\rangle ;\ \ \left( j=1,2,3,4\right) ,  \notag \\
\mathcal{K}_{2} &=&\frac{1}{\mu _{2}\rho ^{2}}\langle \mathbf{V}%
_{2}^{^{\prime }},\mathbf{V}_{3}\rangle =\frac{1}{\mu _{3}\rho ^{2}}\langle
\mathbf{V}_{3}^{^{\prime }},\mathbf{V}_{2}\rangle ,  \notag \\
\mathcal{K}_{3} &=&\frac{1}{\mu _{4}\rho ^{2}}\langle \mathbf{V}%
_{3}^{^{\prime }},\mathbf{V}_{4}\rangle =\frac{1}{\mu _{5}\rho ^{2}}\langle
\mathbf{V}_{4}^{^{\prime }},\mathbf{V}_{3}\rangle .
\end{eqnarray}%
\textbf{Timelike curves:}\newline
Now, if we consider $\gamma :I\rightarrow E_{1}^{4}$ be \textbf{a timelike
curve} parametrized by arclength $s.$ Then as above, we can write%
\begin{eqnarray}
\mathbf{U}_{1} &=&\rho \mathbf{T},\   \notag \\
\mathbf{U}_{2} &=&\rho \mathbf{N},\   \notag \\
\ \ \mathbf{U}_{3} &=&\rho \mathbf{B}_{1},\   \notag \\
\ \mathbf{U}_{4} &=&\rho \mathbf{B}_{2}.
\end{eqnarray}%
Thus, $\left\{ \mathbf{U}_{1},\mathbf{U}_{2},\ \mathbf{U}_{3},\ \mathbf{U}%
_{4}\right\} $ is an equiform invariant tetrahedron of the curve $\gamma $
\cite{7}$.$

The derivatives of these vectors with respect to $\sigma $ are as follows:
\begin{eqnarray}
\mathbf{U}_{1}^{\prime } &=&\frac{d}{d\sigma }\left( \mathbf{U}_{1}\right)
=\rho \frac{d}{ds}\left( \rho \mathbf{T}\right) =\dot{\rho}\mathbf{U}_{1}+%
\mathbf{U}_{2},  \notag \\
\mathbf{U}_{2}^{^{\prime }} &=&\frac{d}{d\sigma }\left( \mathbf{U}%
_{2}\right) =\rho \frac{d}{ds}\left( \rho \mathbf{N}\right) =\mathbf{U}_{1}+%
\dot{\rho}\mathbf{U}_{2}+\left( \frac{\bar{\kappa}_{2}}{\bar{\kappa}_{1}}%
\right) \mathbf{U}_{3},  \notag \\
\mathbf{U}_{3}^{^{\prime }} &=&\frac{d}{d\sigma }\left( \mathbf{U}%
_{3}\right) =\rho \frac{d}{ds}\left( \rho \mathbf{B}_{1}\right) =-\left(
\frac{\bar{\kappa}_{2}}{\bar{\kappa}_{1}}\right) \mathbf{U}_{2}+\dot{\rho}%
\mathbf{U}_{3}+\left( \frac{\bar{\kappa}_{3}}{\bar{\kappa}_{1}}\right)
\mathbf{U}_{4},  \notag \\
\mathbf{U}_{4}^{^{\prime }} &=&\frac{d}{d\sigma }\left( \mathbf{U}%
_{4}\right) =\rho \frac{d}{ds}\left( \rho \mathbf{B}_{2}\right) =-\left(
\frac{\bar{\kappa}_{3}}{\bar{\kappa}_{1}}\right) \mathbf{U}_{3}+\dot{\rho}%
\mathbf{U}_{4}.
\end{eqnarray}%
Hence, the Frenet formulas in the equiform geometry of the Minkowski $4$%
-space can be wriatten as%
\begin{eqnarray}
\mathbf{U}_{1}^{^{\prime }} &=&\mathcal{\bar{K}}_{1}\mathbf{U}_{1}+\mathbf{U}%
_{2},  \notag \\
\mathbf{U}_{2}^{^{\prime }} &=&\mathbf{U}_{1}+\mathcal{\bar{K}}_{1}\mathbf{U}%
_{2}+\mathcal{\bar{K}}_{2}\mathbf{U}_{3},  \notag \\
\mathbf{U}_{3}^{^{\prime }} &=&-\mathcal{\bar{K}}_{2}\mathbf{U}_{2}+\mathcal{%
\bar{K}}_{1}\mathbf{U}_{3}+\mathcal{\bar{K}}_{3}\mathbf{U}_{4},  \notag \\
\mathbf{U}_{4}^{^{\prime }} &=&-\mathcal{\bar{K}}_{3}\mathbf{U}_{3}+\mathcal{%
\bar{K}}_{1}\mathbf{U}_{4}.
\end{eqnarray}%
The functions $\mathcal{\bar{K}}_{1},\ \mathcal{\bar{K}}_{2},\ \mathcal{\bar{%
K}}_{3}$ are the equiform curvatures. In a matrix form, we have%
\begin{equation}
\left[
\begin{array}{c}
\mathbf{U}_{1}^{^{\prime }} \\
\mathbf{U}_{2}^{^{\prime }} \\
\mathbf{U}_{3}^{^{\prime }} \\
\mathbf{U}_{4}^{^{\prime }}%
\end{array}%
\right] =\left[
\begin{array}{cccc}
\mathcal{\bar{K}}_{1} & 1 & 0 & 0 \\
1 & \mathcal{\bar{K}}_{1} & \mathcal{\bar{K}}_{2} & 0 \\
0 & -\mathcal{\bar{K}}_{2} & \mathcal{\bar{K}}_{1} & \mathcal{\bar{K}}_{3}
\\
0 & 0 & -\mathcal{\bar{K}}_{3} & \mathcal{\bar{K}}_{1}%
\end{array}%
\right] \left[
\begin{array}{c}
\mathbf{U}_{1} \\
\mathbf{U}_{2} \\
\mathbf{U}_{3} \\
\mathbf{U}_{4}%
\end{array}%
\right] ,
\end{equation}%
where%
\begin{eqnarray}
\mathcal{\bar{K}}_{1} &=&\frac{1}{\rho ^{2}}\langle \mathbf{U}_{j}^{^{\prime
}},\mathbf{U}_{j}\rangle ;\ \ \left( j=1,2,3,4\right) ,  \notag \\
\mathcal{\bar{K}}_{2} &=&\frac{1}{\rho ^{2}}\langle \mathbf{U}_{2}^{^{\prime
}},\mathbf{U}_{3}\rangle =-\frac{1}{\rho ^{2}}\langle \mathbf{U}%
_{3}^{^{\prime }},\mathbf{U}_{2}\rangle ,  \notag \\
\mathcal{\bar{K}}_{3} &=&\frac{1}{\rho ^{2}}\langle \mathbf{U}_{3}^{^{\prime
}},\mathbf{U}_{4}\rangle =-\frac{1}{\rho ^{2}}\langle \mathbf{U}%
_{4}^{^{\prime }},\mathbf{U}_{3}\rangle .\ \ \ \ \ \ \ \
\end{eqnarray}

\section{The characterizations of space and timelike curves in $E_{1}^{4}$}

In this section, we characterize the space and timelike curves using their
equiform curvatures $\mathcal{K}_{i}\ ,\mathcal{\bar{K}}_{i}\ \left(
i=1,2,3\right) $ in $E_{1}^{4}$ which have important geometric
interpretation, as follows:

\begin{enumerate}
\item If $\mathcal{K}_{2}=const.,\ \mathcal{K}_{3}=const.,\ $of a spacelike
curve, then the curve is a general helix and vice versa. Here, we do not set
conditions on $\mathcal{K}_{1}.$

\item If the above condition holds and $\mathcal{K}_{1}$ is identically
zero, then the spacelike curve is a W-curve (Since all three curvatures $%
\kappa _{1},\kappa _{2},\kappa _{3}$ are constant. For more details, see
\cite{8}).
\end{enumerate}

\begin{theorem}
Let $\alpha $ be a spacelike curve in $E_{1}^{4}$ with the equiform
invariant tetrahedron $\left\{ \mathbf{V}_{1},\mathbf{V}_{2},\ \mathbf{V}%
_{3},\ \mathbf{V}_{4}\right\} \ $and equiform curvature $\mathcal{K}_{1}\neq
0.$ Then $\alpha $ has $\mathcal{K}_{2}\equiv 0$ if and only if $\alpha $
lies fully in a $2$-dimensional subspace of $E_{1}^{4}.$
\end{theorem}

\begin{theorem}
Let $\alpha $ be a spacelike curve in $E_{1}^{4}$ and $\left\{ \mathbf{V}%
_{1},\mathbf{V}_{2},\ \mathbf{V}_{3},\ \mathbf{V}_{4}\right\} $is the
equiform invariant tetrahedron of it. When $\mathcal{K}_{1},\mathcal{K}%
_{2}\neq 0,$ then $\alpha $ has $\mathcal{K}_{3}\equiv 0$ if and only if $%
\alpha $ lies fully in a hyperplane of $E_{1}^{4}.$
\end{theorem}

\begin{proof}
If $\alpha \ $has $\mathcal{K}_{3}\equiv 0,$ then from $(3.9)$, we have%
\begin{eqnarray*}
\alpha ^{^{\prime }} &=&\mathbf{V}_{1}, \\
\alpha ^{^{\prime \prime }} &=&\mathcal{K}_{1}\mathbf{V}_{1}+\mathbf{V}_{2},
\\
\alpha ^{^{\prime \prime \prime }} &=&(\rho \mathcal{\dot{K}}_{1}+\mathcal{K}%
_{1}^{2}+\mu _{1})\mathbf{V}_{1}+2\mathcal{K}_{1}\mathbf{V}_{2}+\mu _{2}%
\mathcal{K}_{2}\mathbf{V}_{3}, \\
\alpha ^{\left( 4\right) } &=&\frac{d\left( \rho \mathcal{\dot{K}}_{1}+%
\mathcal{K}_{1}^{2}+\mu _{1}\right) }{d\sigma }+\mathcal{K}_{1}\left( \rho
\mathcal{\dot{K}}+\mathcal{K}_{1}^{2}+\mu _{1}\right) +2\mu _{1}\mathcal{K}%
_{1}\mathbf{V}_{1} \\
&&+\left( 3\rho \mathcal{\dot{K}}_{1}+3\mathcal{K}_{1}^{2}+\mu _{2}\mu _{3}%
\mathcal{K}_{2}^{2}+\mu _{1}\right) \mathbf{V}_{2} \\
&&+\left( 3\mu _{2}\mathcal{K}_{1}\mathcal{K}_{2}+\mu _{2}\rho \mathcal{\dot{%
K}}_{2}\right) \mathbf{V}_{3}+\left( \mu _{2}\mu _{4}\mathcal{K}_{2}\mathcal{%
K}_{3}\right) \mathbf{V}_{4}.
\end{eqnarray*}%
Hence, by using Maclaurin expansion for $\alpha $, given by%
\begin{equation*}
\alpha \left( \sigma \right) =\alpha \left( 0\right) +\alpha ^{^{\prime
}}\left( 0\right) \sigma +\alpha ^{^{\prime \prime }}\left( 0\right) \frac{%
\sigma ^{2}}{2!}+\alpha ^{^{\prime \prime \prime }}\left( 0\right) \frac{%
\sigma ^{3}}{3!}+\alpha ^{\left( 4\right) }\left( 0\right) \frac{\sigma ^{4}%
}{4!}+...,
\end{equation*}%
we obtain that $\alpha $ lies fully in a spacelike hyperplane of the space $%
E_{1}^{4}$ by spanned%
\begin{equation*}
\left\{ \alpha ^{^{\prime }}\left( 0\right) ,\alpha ^{^{\prime \prime
}}\left( 0\right) ,\alpha ^{^{\prime \prime \prime }}\left( 0\right)
\right\} .
\end{equation*}%
Conversely, assume that $\alpha $ satisfies the assumptions of the theorem
and lies fully in a spacelike hyperplane $\Pi $ of the space $E_{1}^{4}.$
Then, there exist the points $p,q\in E_{1}^{4},$ such that $\alpha $
satisfies the equation of $\Pi $ given by%
\begin{equation}
\langle \alpha \left( \sigma \right) -p,q\rangle =0,
\end{equation}%
where $q\in \Pi ^{\perp }$ is a timelike vector$.$ Differentiation of the
last equation yields%
\begin{equation*}
\langle \alpha ^{^{\prime }},q\rangle =\langle \alpha ^{^{\prime \prime
}},q\rangle =\langle \alpha ^{^{\prime \prime \prime }},q\rangle =0.
\end{equation*}%
Therefore, $\alpha ^{\prime },\alpha ^{\prime \prime },\alpha ^{\prime
\prime \prime }\in \Pi .$ Now, since%
\begin{equation*}
\alpha ^{^{\prime }}=\mathbf{V}_{1}\text{ \ and \ }\alpha ^{^{\prime \prime
}}=\mathcal{K}_{1}\mathbf{V}_{1}+\mathbf{V}_{2},
\end{equation*}%
it follows that%
\begin{equation}
\langle \mathbf{V}_{1},q\rangle =\langle \mathbf{V}_{2},q\rangle =0.
\end{equation}%
Next, differentiation $(4.2)$ gives%
\begin{equation}
\langle \mathbf{V}_{2}^{\prime },q\rangle =0.
\end{equation}%
From the equiform Frenet equations, we obtain
\begin{equation}
\langle \mathbf{V}_{3},q\rangle =0.
\end{equation}%
Again, differentiating $(4.4)$ leads to
\begin{equation*}
\langle \mu _{3}\mathcal{K}_{2}\mathbf{V}_{2}+\mathcal{K}_{1}\mathbf{V}%
_{3}+\mu _{4}\mathcal{K}_{3}\mathbf{V}_{4},q\rangle =\mathcal{K}_{3}\langle
\mathbf{V}_{4},q\rangle =0.
\end{equation*}%
Because $\mathbf{V}_{4}$ is the only vector perpendicular to $\left\{
\mathbf{V}_{1},\mathbf{V}_{2},\ \mathbf{V}_{3}\right\} ,$ we obtain%
\begin{equation*}
\mathcal{K}_{3}=0.
\end{equation*}%
This completes the proof.
\end{proof}

\begin{theorem}
Let $\alpha $ be a spacelike curve with equiform invariant vector $\mathbf{V}%
_{3}$ in the equiform geometry of $E_{1}^{4}$. Then, the curve $\alpha $ is
a general helix if and only if
\begin{equation}
\mathbf{V}_{3}^{^{\prime \prime }}+\psi _{1}\mathbf{V}_{3}=\psi _{2}\mathbf{V%
}_{1}+\psi _{3}\mathbf{V}_{2}+\psi _{4}\mathbf{V}_{4},
\end{equation}%
where
\begin{eqnarray*}
\psi _{1} &=&-\left( \rho \mathcal{\dot{K}}_{1}+\mathcal{K}_{1}^{2}+\mu
_{2}\mu _{3}\mathcal{K}_{2}^{2}+\mu _{4}\mu _{5}\mathcal{K}_{3}^{2}\right) ,
\\
\psi _{2} &=&\mu _{1}\mu _{3}\mathcal{K}_{2},\  \\
\psi _{3} &=&2\mu _{3}\mathcal{K}_{1}\mathcal{K}_{2}, \\
\psi _{4} &=&2\mu _{4}\mathcal{K}_{1}\mathcal{K}_{3}.
\end{eqnarray*}
\end{theorem}

\begin{proof}
Suppose that the curve $\alpha $ is a general helix. Thus, we have%
\begin{equation}
\mathcal{K}_{2}=const.\text{ and }\mathcal{K}_{3}=const.
\end{equation}%
From $(3.9)$ and $(4.6)$ , it is easy to prove that the equation $(4.5)$ is
satisfied.\newline
Conversely, we assume that the equation $(4.5)$ holds. Then from $(3.9)$, it
follows that%
\begin{equation}
\mathbf{V}_{3}^{^{\prime }}=\mu _{3}\mathcal{K}_{2}\mathbf{V}_{2}+\mathcal{K}%
_{1}\mathbf{V}_{3}+\mu _{4}\mathcal{K}_{3}\mathbf{V}_{4}.
\end{equation}%
Differentiating $(4.7)$ with respect to $\sigma $, we get%
\begin{eqnarray*}
\mathbf{V}_{3}^{^{\prime \prime }} &=&\left( \mu _{1}\mu _{3}\mathcal{K}%
_{2}\right) \mathbf{V}_{1}+\left( \mu _{3}\rho \mathcal{\dot{K}}_{2}+2\mu
_{3}\mathcal{K}_{1}\mathcal{K}_{2}\right) \mathbf{V}_{2} \\
&&+\left( \rho \mathcal{\dot{K}}_{1}+\mathcal{K}_{1}^{2}+\mu _{2}\mu _{3}%
\mathcal{K}_{2}^{2}+\mu _{4}\mu _{5}\mathcal{K}_{3}^{2}\right) \mathbf{V}_{3}
\\
&&+\left( \mu _{4}\rho \mathcal{\dot{K}}_{3}+2\mu _{4}\mathcal{K}_{1}%
\mathcal{K}_{3}\right) \mathbf{V}_{4},
\end{eqnarray*}%
so, we obtain%
\begin{equation*}
\mathcal{\dot{K}}_{2}=0\text{ and }\mathcal{\dot{K}}_{3}=0,
\end{equation*}%
which completes the proof.
\end{proof}

\begin{theorem}
Let $\gamma $ be \textbf{a timelike} curve in $E_{1}^{4}$ with the equiform
invariant tetrahedron $\left\{ \mathbf{U}_{1},\mathbf{U}_{2},\ \mathbf{U}%
_{3},\ \mathbf{U}_{4}\right\} \ $and with equiform curvatures $\mathcal{\bar{%
K}}_{1}\neq 0.$ Then $\gamma $ has $\mathcal{\bar{K}}_{2}\equiv 0$ if and
only if $\gamma $ lies fully in a 2-dimensional subspace of $E_{1}^{4}.$
\end{theorem}

\begin{theorem}
Let $\gamma $ be \textbf{a timelike} curve in $E_{1}^{4}$ with the equiform
invariant tetrahedron $\left\{ \mathbf{U}_{1},\mathbf{U}_{2},\ \mathbf{U}%
_{3},\ \mathbf{U}_{4}\right\} \ $and with equiform curvatures $\mathcal{\bar{%
K}}_{1},\mathcal{\bar{K}}_{2}\neq 0.$ Then $\gamma $ has $\mathcal{\bar{K}}%
_{3}\equiv 0$ if and only if $\gamma $ lies fully in a hyperplane of $%
E_{1}^{4}.$
\end{theorem}

\begin{proof}
If $\gamma \ $has $\mathcal{\bar{K}}_{3}\equiv 0,$ then by using the
equiform Frenet equations $(3.13)$ we obtain%
\begin{eqnarray*}
\gamma ^{^{\prime }} &=&\mathbf{U}_{1}, \\
\gamma ^{^{\prime \prime }} &=&\mathcal{\bar{K}}_{1}\mathbf{U}_{1}+\mathbf{U}%
_{2}, \\
\gamma ^{^{\prime \prime \prime }} &=&\left( \rho \overset{\mathbf{\cdot }}{%
\mathcal{\bar{K}}_{1}}+\mathcal{\bar{K}}_{1}^{2}+1\right) \mathbf{U}_{1}+2%
\mathcal{\bar{K}}_{1}\mathbf{U}_{2}+\mathcal{\bar{K}}_{2}\mathbf{U}_{3}, \\
\gamma ^{\left( 4\right) } &=&\left( \frac{d\left( \rho \overset{\mathbf{%
\cdot }}{\mathcal{\bar{K}}_{1}}+\mathcal{\bar{K}}_{1}^{2}+1\right) }{d\sigma
}+\mathcal{\bar{K}}_{1}\left( \rho \overset{\mathbf{\cdot }}{\mathcal{\bar{K}%
}_{1}}+\mathcal{\bar{K}}_{1}^{2}+1\right) +2\mathcal{\bar{K}}_{1}\right)
\mathbf{U}_{1} \\
&&+\left( 3\rho \overset{\mathbf{\cdot }}{\mathcal{\bar{K}}_{1}}+3\mathcal{%
\bar{K}}_{1}^{2}-\mathcal{\bar{K}}_{2}^{2}+1\right) \mathbf{U}_{2} \\
&&+\left( 3\mathcal{\bar{K}}_{1}\mathcal{\bar{K}}_{2}+\rho \overset{\mathbf{%
\cdot }}{\mathcal{\bar{K}}_{2}}\right) \mathbf{U}_{3}+\mathcal{\bar{K}}_{2}%
\mathcal{\bar{K}}_{3}\mathbf{U}_{4}.
\end{eqnarray*}%
Next, all higher order derivatives of $\gamma $ are linear combinations of
vectors $\gamma ^{^{\prime }},\gamma ^{^{\prime \prime }},\gamma ^{^{\prime
\prime \prime }}$, so by using Maclaurin expansion for $\gamma $ given by%
\begin{equation*}
\gamma \left( \sigma \right) =\gamma \left( 0\right) +\gamma ^{^{\prime
}}\left( 0\right) \sigma +\gamma ^{^{\prime \prime }}\left( 0\right) \frac{%
\sigma ^{2}}{2!}+\gamma ^{^{\prime \prime \prime }}\left( 0\right) \frac{%
\sigma ^{3}}{3!}+\gamma ^{\left( 4\right) }\left( 0\right) \frac{\sigma ^{4}%
}{4!}+...,
\end{equation*}%
we conclude that $\gamma $ lies fully in a timelike hyperplane of the space $%
E_{1}^{4},$ spanned by%
\begin{equation*}
\left\{ \gamma ^{^{\prime }}\left( 0\right) ,\gamma ^{^{\prime \prime
}}\left( 0\right) ,\gamma ^{^{\prime \prime \prime }}\left( 0\right)
\right\} .
\end{equation*}%
Conversely, we suppose that $\gamma $ lies fully in a timelike hyperplane $%
\bar{\Pi}$ of the space $E_{1}^{4}.$ Then, there exist the points $p,q\in
E_{1}^{4}$ such that $\gamma $ satisfies the equation of $\bar{\Pi}$ given by%
\begin{equation}
\langle \gamma \left( \sigma \right) -p,q\rangle =0,
\end{equation}%
where $q\in \bar{\Pi}^{\perp }.$ By differentiating $\left( 4.8\right) $
with respect to $\sigma $, we can write%
\begin{equation*}
\langle \gamma ^{^{\prime }},q\rangle =\langle \gamma ^{^{\prime \prime
}},q\rangle =\langle \gamma ^{^{\prime \prime \prime }},q\rangle =0.
\end{equation*}%
Since%
\begin{equation*}
\gamma ^{^{\prime }}=\mathbf{U}_{1}\text{ \ and \ }\gamma ^{^{\prime \prime
}}=\mathcal{\bar{K}}_{1}\mathbf{U}_{1}+\mathbf{U}_{2},
\end{equation*}%
it follows that%
\begin{equation}
\langle \mathbf{U}_{1},q\rangle =\langle \mathbf{U}_{2},q\rangle =0.
\end{equation}%
Similarly, we have%
\begin{equation}
\langle \mathbf{U}_{3},q\rangle =0.
\end{equation}%
Again, differentiation $(4.10)$ gives%
\begin{equation*}
\langle -\mathcal{\bar{K}}_{2}\mathbf{U}_{2}+\mathcal{\bar{K}}_{1}\mathbf{U}%
_{3}+\mathcal{\bar{K}}_{3}\mathbf{U}_{4},q\rangle =\mathcal{\bar{K}}%
_{3}\langle \mathbf{U}_{4},q\rangle =0,
\end{equation*}%
because $\mathbf{U}_{4}$ is the only vector perpendicular to $\left\{
\mathbf{U}_{1},\mathbf{U}_{2},\ \mathbf{U}_{3}\right\} ,$ we obtain%
\begin{equation*}
\mathcal{\bar{K}}_{3}=0,
\end{equation*}%
this completes the proof.
\end{proof}

\begin{theorem}
Let $\gamma $ be \textbf{a timelike curve }with equiform invariant vector $%
\mathbf{U}_{3}$ in the equiform geometry of $E_{1}^{4}$. Then $\gamma $ is a
general helix if and only if%
\begin{equation}
\mathbf{U}_{3}^{^{\prime \prime }}+\varphi _{1}\mathbf{U}_{3}=\varphi _{2}%
\mathbf{U}_{1}+\varphi _{3}\mathbf{U}_{2}+\varphi _{4}\mathbf{U}_{4},
\end{equation}%
where
\begin{eqnarray*}
\varphi _{1} &=&-\rho \overset{\cdot }{\mathcal{\bar{K}}_{1}}-\mathcal{\bar{K%
}}_{1}^{2}+\mathcal{\bar{K}}_{2}^{2}+\mathcal{\bar{K}}_{3}^{2}, \\
\varphi _{2} &=&-\mathcal{\bar{K}}_{2},\  \\
\varphi _{3} &=&-2\mathcal{\bar{K}}_{1}\mathcal{\bar{K}}_{2}, \\
\varphi _{4} &=&2\mathcal{\bar{K}}_{1}\mathcal{\bar{K}}_{3}.
\end{eqnarray*}
\end{theorem}

\begin{proof}
Suppose that the curve $\gamma $ is a general helix. Thus, we have%
\begin{equation}
\mathcal{\bar{K}}_{2}=const.\text{ and }\mathcal{\bar{K}}_{3}=const.
\end{equation}%
From $(3.13)$ and $(4.9)$ , it is easy to prove that the equation $(4.11)$
is satisfied.\newline
Conversely, we assume that the equation $(4.11)$ holds. Then from $(3.13)$,
it follows that%
\begin{equation}
\mathbf{U}_{3}^{^{\prime }}=-\mathcal{\bar{K}}_{2}\mathbf{U}_{2}+\mathcal{%
\bar{K}}_{1}\mathbf{U}_{3}+\mathcal{\bar{K}}_{3}\mathbf{U}_{4},
\end{equation}%
By differentiating $(4.9)$ with respect to $\sigma $, we get%
\begin{eqnarray*}
\mathbf{U}_{3}^{^{\prime \prime }} &=&-\mathcal{\bar{K}}_{2}\mathbf{U}%
_{1}-(\rho \overset{\cdot }{\mathcal{\bar{K}}_{2}}+2\mathcal{\bar{K}}_{1}%
\mathcal{\bar{K}}_{2})\mathbf{U}_{2} \\
&&+(\rho \overset{\cdot }{\mathcal{\bar{K}}_{1}}+\mathcal{\bar{K}}_{1}^{2}-%
\mathcal{\bar{K}}_{2}^{2}-\mathcal{\bar{K}}_{3}^{2})\mathbf{U}_{3} \\
&&+(\rho \overset{\cdot }{\mathcal{\bar{K}}_{3}}+2\mathcal{\bar{K}}_{1}%
\mathcal{\bar{K}}_{3})\mathbf{U}_{4},
\end{eqnarray*}%
so, we obtain%
\begin{equation*}
\overset{\cdot }{\mathcal{\bar{K}}_{2}}=0\text{ and }\overset{\cdot }{%
\mathcal{\bar{K}}_{3}}=0.
\end{equation*}%
It completes the proof.
\end{proof}

\section{Conclusion}

In the 4-dimensional Minkowski space, equiform differential geometry of
space and timelike curves are investigated. Frenet formulas in the equiform
differential geometry of the Minkowski $4$- space $E_{1}^{4}$ for these
curves are obtained. Moreover, we have characterized these curves using their
equiform curvatures $\mathcal{K}_{i}\ ,\mathcal{\bar{K}}_{i}\ \left(
i=1,2,3\right) $ in $E_{1}^{4}$ which have important geometric
interpretation.

\end{document}